\documentclass[ECP]{ejpecp} 
\pdfoutput=1

\usepackage[T1]{fontenc}
\usepackage[utf8]{inputenc}

\usepackage{ifthen}

\AUTHORS{%
Michel~Bena\"{i}m\footnote{Universit\'e de Neuch\^atel, Suisse.   
\EMAIL{michel.benaim@unine.ch}}
\and St\'ephane~Le Borgne\footnote{
Universit\'e de Rennes 1, France.
\EMAIL{{stephane.leborgne,florent.malrieu}@univ-rennes1.fr}}
\and Florent~Malrieu\footnotemark[2]
\and Pierre-Andr\'e~Zitt\footnote{
Universit\'e de Bourgogne, France.
\EMAIL{pierre-andre.zitt(AT)u-bourgogne.fr}
}
}

\SHORTTITLE{Quantitative ergodicity for some switched dynamical systems}

\TITLE{Quantitative ergodicity for some switched dynamical systems}


\AMSSUBJ{60J75; 60J25; 93E15; 34D23}

\SUBMITTED{April 9, 2012} 
\ACCEPTED{November 15, 2012} 

\ARXIVID{1204.1922v3}  

\VOLUME{17}
\YEAR{2012}
\PAPERNUM{56}
\DOI{v17-1932}

\SUBMITTED{January 1, 2012} 
\ACCEPTED{January 1, 2012} 

\ARXIVID{1204.1922v3}

\ABSTRACT{
We provide quantitative bounds for the long time behavior of a class 
of Piecewise Deterministic Markov Processes with state space 
$\mathbb{R}^d\times E$ where $E$ is a finite set. The continuous 
component evolves according to a smooth vector field that switches 
at the jump times of the discrete coordinate. The jump rates may depend 
on the whole position of the process. Under regularity assumptions on the 
jump rates and stability conditions for the vector fields we provide 
explicit exponential upper bounds for the convergence to equilibrium 
in terms of Wasserstein distances. As an example, we obtain convergence
results for a stochastic version of the Morris--Lecar model of neurobiology.
}

\newcommand{\ind}{\mathds{1}}
\newtheorem{assumption}[theorem]{Assumption}%

\newcommand{\dE}{\mathbb{E}}

\newcommand{\dP}{\mathbb{P}}
\newcommand{\dR}{\mathbb{R}}

\newcommand{\cE}{\mathcal{E}}
\newcommand{\cG}{\mathcal{G}}

\newcommand{\cL}{\mathcal{L}}

\newcommand{\cP}{\mathcal{P}}

\newcommand{\cW}{\mathcal{W}}

\newcommand{\gen}{\mathfrak{L}}
 
\newcommand{\ABS}[1]{{{\left| #1 \right|}}} 
\newcommand{\BRA}[1]{{{\left\{#1\right\}}}} 
\newcommand{\DP}[1]{{{\left<#1\right>}}} 
\newcommand{\NRM}[1]{{{\left\| #1\right\|}}} 
\newcommand{\PAR}[1]{{{\left(#1\right)}}} 
\newcommand{\SBRA}[1]{{{\left[#1\right]}}} 
\renewcommand{\leq}{\leqslant}
\renewcommand{\geq}{\geqslant}

\newcommand{\pullbk}[2]{%
\ifthenelse{\equal{#1}{}}{%
  \Phi^{\star}_{#2}}{%
  \Phi^{#1,\star}_{#2}}%
}
\newcommand{\pushfw}[2]{%
\ifthenelse{\equal{#2}{}}{%
  \Phi^{#1}_{\star}}{%
  \Phi^{#1}_{#2,\star}}%
}

\begin{document}

\section{Introduction and main results}

Piecewise deterministic Markov processes (PDMPs in short) are 
intensively used in many applied areas (molecular biology \cite{RMC}, 
storage modelling \cite{boxma}, Internet traffic
\cite{MR1895332,MR2588247,GR2}, neuronal
  activity \cite{PTW,riedler}, populations growth models \cite{Kussel}...). 
 Roughly speaking, a Markov process is a PDMP if its randomness is only 
 given by the jump mechanism: in particular, it admits no diffusive 
 dynamics. This huge class of processes has been introduced by 
 Davis (see \cite{davis,MR1283589}) in a general framework. Several 
 works \cite{MR1039184,MR1710229,MR2274799} deal with their 
 long time behavior (existence of an invariant probability measure, 
 Harris recurrence, exponential ergodicity...). In particular, it is shown in 
 \cite{MR2385873} that the behavior of a general PDMP can be related to 
 the one of the discrete time Markov chain made of the positions at the 
 jump times of the process and of an additional independent Poisson 
 process. Nevertheless, this general approach does not seem to 
 provide quantitative bounds for the convergence 
 to equilibrium. Recent papers have tried to establish such 
 estimates for some specific PDMPs (see \cite{CMP,FGM,BCGMZ}) 
 or continuous time Markov chains (see \cite{MR2548501}).

 In the present paper, we investigate the long time behavior of an 
 interesting subclass of PDMPs that plays a role in molecular biology 
 (see \cite{ML81,RMC, riedler,WTP10}). We consider a PDMP on $\dR^d\times E$ 
 where $E$ is a finite set. The first coordinate moves continuously 
 on $\dR^d$ according to a smooth vector field that depends on the 
 second coordinate whereas the second coordinate jumps with a 
 rate that may depend on the first one. This class of Markov processes 
 is reminiscent of the so-called iterated random functions in the discrete 
 time setting (see \cite{diaconis} for a good review of this topic). 
   
Let $E$ be a finite set, $(a(\cdot,i,j))_{i,j\in E^2}$ be 
$n^2$ nonnegative continuous functions on $\dR^d$, and, for
any $i\in E$, $F^i\,:\, \dR^d\mapsto \dR^d$ 
 be a smooth vector field such that the ordinary differential equation 
\[
\begin{cases}
 x_t'=F^i(x_t),\quad t>0;\\
 x_0=x,
\end{cases}
\]
 has a unique and global solution $t\mapsto \varphi^i_t(x)$ 
 on $[0,\infty)$ for any initial condition $x\in\dR^d$.
Let us consider the Markov process 
\[
{(Z_t)}_{t\geq 0}={((X_t,I_t))}_{t\geq 0} 
\text{ on }\dR^d\times E 
\]
defined by its extended generator $L$ as follows:
\begin{equation}\label{eq:gen-inf-var}
Lf(x,i)=
\DP{F^i(x),\nabla_x f(x,i)}+
\sum_{j\in E}a(x,i,j)(f(x,j)-f(x,i)) 
\end{equation}
for any smooth function $f:\ \dR^d\times E\rightarrow \dR$ 
(see \cite{MR1283589} for full details on the domain of $L$). 

\begin{definition}
In the sequel, $\eta_t$ (resp. $\mu_t$) stands for the law of $Z_t=(X_t,I_t)$
(resp of $X_t$) if ${(Z_t)}_{t\geq 0}$ is driven by~\eqref{eq:gen-inf-var} with 
initial distribution~$\eta_0$.
\end{definition}

A simple case occurs when $a(x,i,j)$ can be written as $\lambda(x,i)P(i,j)$
for $(\lambda(\cdot, i))_{i\in E}$ a set of nonnegative continuous functions, 
and $P$ an irreducible stochastic matrix, in which case:
\begin{equation}\label{eq:gen-inf}
Lf(x,i)=
\DP{F^i(x),\nabla_x f(x,i)}+
\lambda(x,i)\sum_{j\in E}P(i,j)(f(x,j)-f(x,i)).
\end{equation}
Let us describe the dynamics of this process in this simple case, the general case
being similar. Assume that 
$(X_0,I_0)=(x,i)\in\dR^d\times E$. Before the first jump time $T_1$ of $I$, 
the first component $X$ is driven by the vector field $F^i$ and then 
$X_t=\varphi^i_t(x)$. The time $T_1$ can be defined by: 
\[
T_1=\inf\BRA{t>0\, :\ \int_0^{t}\!\lambda(X_s,i)\,ds\geq E_1},
\]
where $E_1$ is an exponential random variable with parameter 1. 
Since the paths of $X$ are deterministic between the jump times 
of $I$, the randomness of $T_1$ comes from the one of $E_1$ 
and 
\[
T_1=\inf\BRA{t>0\, :\ \int_0^t\!\lambda(\varphi^i_s(x),i)\,ds\geq E_1}.
\]

\begin{remark}
Notice that $\dP_{(x,i)}(T_1=+\infty)>0$ if and only if 
\[
 \int_0^{+\infty}\!\lambda(\varphi^i_s(x),i)\,ds<+\infty.
\]
If we assume that $\underline\lambda:=\inf_{(x,i)}\lambda(x,i)>0$ then 
the process $I$ jumps infinitely often. 
\end{remark}

At time $T_1$, the second coordinate $I$ performs a jump with 
the law $P(i,\cdot)$ and the vector field that drives the evolution 
of $X$ is switched\ldots 

\begin{remark}
In general, $I$ is not a Markov process on its own since its jump 
rates depend on $X$. In this paper, we will study both the 
simple --- Markovian --- case and the general case. 
\end{remark}

The main goal of the present work is to provide quantitative bounds for 
the long time behavior of ergodic processes driven by \eqref{eq:gen-inf} 
thanks to the construction of explicit couplings. We need to introduce a 
distance on the set of probability measures to quantify this convergence. 

\begin{definition}
Let $\cP_p$ be the set of probability measures on $\dR^d\times E$ such that the first 
marginal has a finite $p^\mathrm{th}$ moment. If $\eta\in\cP_p$ and $Z$ is a random 
variable on $\dR^d\times E$ with distribution $\eta$, we denote by $X$ and $I$ the 
two coordinates of $Z$ in $\dR^d$ and $E$ respectively and by $\mu$ and $\nu$ 
their distributions.

Let us now define the distance $\cW_p$ on $\cP_p$ as follows: for $\eta,\tilde\eta\in\cP_p$, 
\[
\cW_p(\eta,\tilde\eta)=\inf\PAR{\PAR{\dE\SBRA{\vert X-\tilde X\vert^p}}^{1/p}+\dP(I\neq \tilde I)
\,:\, (X,I)\sim \eta\text{ and }(\tilde X,\tilde I)\sim \tilde \eta}.
\] 
\end{definition}

The distance $\cW_p$ is made of a mixture of the Wasserstein distance for the 
first component and the total variation distance for the second one. Recall that 
for every $p\geq1$, the Wasserstein distance $W_p$ between two laws 
$\mu$ and $\tilde \mu$ on $\dR^d$ with finite $p^\mathrm{th}$ moment is defined by
\begin{equation}\label{eq:Wp}
  W_p(\mu,\tilde\mu) %
  = \left(\inf_{\Pi}\int_{\dR^d\times\dR^d}\!|x-\tilde x|^p\,\Pi(dx,d\tilde x)
  \right)^{1/p}
\end{equation}
where the infimum runs over all the probability measures $\Pi$ on 
$\dR^d\times \dR^d$ with marginals $\mu$ and $\tilde \mu$ (such 
measures are called couplings of $\mu$ and $\tilde\mu$). It is well 
known that for any $p\geq1$, the convergence in $W_p$ Wasserstein 
distance is equivalent to weak convergence together 
with convergence of all moments up to order $p$. However, two probability 
measures can be both very close in the $W_p$ sense and singular. Choose
for example $\mu=\delta_0$ and $\tilde \mu=\delta_\varepsilon$. In this 
case, 
\[
W_p(\mu,\tilde\mu)=\varepsilon
\quad\text{and}\quad 
\NRM{\mu-\tilde \mu}_{\mathrm{TV}}=1. 
\]
See e.g. \cite{MR1105086,MR1964483} for further details and properties for 
Wasserstein distances. 

Estimates for the Wasserstein distances do not require (nor provide) any 
information about the support of the invariant measure (which is the set of 
the recurrent points). This set may be difficult to determine and if the initial 
distribution of $X$ is not supported by this set, the law of $X_t$ and the 
invariant measure may be singular. To illustrate this fact, one can consider 
the following trivial example: 
\[
(x,i)\in \dR^2\times\BRA{0,1},\quad \lambda(x,i)=1,\quad F^i(x)=-(x-i a) 
\quad\text{with }a=(1,0). 
\] 
The process $(X,I)$ is ergodic and the first marginal $\mu$ of its 
invariant measure is supported by the segment 
$\BRA{\rho a\ ;\ \rho\in[0,1]}$ (it is shown in \cite{boxma} that $\mu$ 
is a Beta distribution). Despite its extreme simplicity, this process does not
in general converge in total variation: 
if $Z_0=(0,1)$, the law of $X_t$ is singular with the invariant measure 
for any $t\geq 0$, so $\NRM{\cL(X_t)-\mu}_\mathrm{TV}$ is always equal 
to $1$. On the contrary, $W_p(\cL(X_t),\mu)$ goes to $0$ exponentially 
fast (see below).

We are able to get explicit rates of convergence in two situations. 
Firstly, if the jump rates of $I$ does not depend on $X$, 
then the vector fields ${(F^i)}_{i\in E}$ will be assumed to satisfy 
an averaged exponential stability. Secondly, if the jump rates of $I$ 
are assumed to be Lipschitz functions of $X$, then the vector 
fields ${(F^i)}_{i\in E}$ will be assumed to satisfy a uniform 
exponential stability.

\begin{remark}
These conditions are not necessary to ensure the ergodicity of $Z$. 
In~\cite{BH}, it is shown that, for constant jump rates and under Hörmander-like 
conditions on the vectors fields ${(F^i)}_{i\in E}$, the invariant probability measure is unique 
and absolutely continuous with respect to the Lebesgue measure on $\dR^d$.
In a similar setting,  exponential rates of convergence to equilibrium are obtained in~\cite{BLMZ2}. 
But these estimates rely on compactness arguments and are not explicit. 
\end{remark}

%
%
%

\subsection{Constant jump rates}

If the jump rates of $I$ do not depend on $X$, then ${(I_t)}_{t\geq 0}$ is a 
Markov process on the finite space $E$ and ${(X_s)}_{0\leq s\leq t}$ is 
a deterministic function of ${(I_s)}_{0\leq s\leq t}$. Many results are 
available both in the discrete time setting (see 
\cite{kesten,vervaat,goldie-grubel,hitwes,iksanov}) 
and in the continuous time setting (see \cite{guyon,saporta-yao,bgm}).  
Moreover, if Assumption~\ref{as:dissip-moy} (see below) does not hold, \cite{BLMZ1} 
provides a simple example of surprising 
phase transition for a switching of two exponentially stable flows that 
can be explosive (when the jump rates are sufficiently large). 
\begin{assumption}\label{as:ymarkov}
 Assume that the jump rates $(a(\cdot,i,j))_{i\in E}$ do not depend 
 on $x$ and that $I$ is an irreducible Markov process on $E$. Let 
 us denote by $\nu$ its invariant probability measure. 
\end{assumption}
\begin{remark}
  If $a(\cdot,i,j)$ does not depend on $x$, we can always write
  $a(i,j) = \lambda(i)P(i,j)$ with $P(\cdot,\cdot)$ a Markov transition matrix.
\end{remark}

\begin{assumption}\label{as:dissip-moy}
Assume that for any $i\in E$, there exists  $\alpha(i)\in\dR$ such that, 
\[
\DP{x-\tilde x,F^i(x)-F^i(\tilde x)}\leq -\alpha(i)\ABS{x-\tilde x}^2,
\quad
x,\tilde x\in\dR^d,
\]
and that 
\[
\sum_{i\in E}\alpha(i)\nu(i)>0
\]
where $\nu$ is defined in Assumption \ref{as:ymarkov}.
\end{assumption}
Firstly, one can establish that the process $X$ is bounded in some 
$L^p$ space. 
\begin{lemma}\label{lem:moment2}
  Under Assumptions \ref{as:ymarkov} and \ref{as:dissip-moy}, there exists 
 $\kappa>0$ such that, for any $q<\kappa$, the function 
 $t\mapsto\dE(\ABS{X_t}^q)$ is bounded as soon as $\dE(\ABS{X_0}^q)$ 
 is finite. More precisely, there exists $M(q,m)$ such that 
\[
 \sup_{t\geq 0}\dE(\ABS{X_t}^q)\leq M(q,m),
\]
as soon as $\dE(\ABS{X_0}^q)\leq m$. 
\end{lemma}
Let us now turn to the long time behavior estimate. 
\begin{theorem}\label{th:constant}
 Assume that Assumptions \ref{as:ymarkov} and \ref{as:dissip-moy} hold. 
 Let $p<q<\kappa$ and denote by $s$ the conjugate of $q$: 
 $q^{-1}+s^{-1}=1$. Assume that $\mu_0$ and $\tilde \mu_0$ admit a finite 
 $q^{th}$ moment smaller than $m$. Then, 
\[
 \cW_p(\eta_t,\tilde \eta_t)\leq 
 2^{p+1}M(q,m)^{p/q}C_2(p) \exp\PAR{-\frac{\theta_p}{1+s\theta_p/\rho}t},
\] 
where $\rho$ and $\theta_p$ are positive constants depending only on 
the Markov chain $I$. 
\end{theorem}
The constants $\rho$, $\theta_p$ and $C_2(p)$ are given below in
Equations \eqref{eq:saloff}, \eqref{eq:etap} and \eqref{eq:conv-expo}. 
\begin{corollary}
Under the hypotheses of Theorem~\ref{th:constant},
the process $Z$ admits a unique invariant measure $\eta_\infty$ and 
\[
 \cW_p(\eta_t,\eta_\infty)\leq 
 2^{p+1}M(q,m)^{p/q}C_2(p) \exp\PAR{-\frac{\theta_p}{1+s\theta_p/\rho}t}.
\]
\end{corollary}

\subsection{Non constant jump rates}

Let us now turn to the case when the jump rates of $I$ depend on $X$. 
We will assume that  the $a(x,i,j)$ are smooth in the $x$ variable and that each vector field 
$F^i$ has a unique stable point.  
\begin{assumption}\label{ass:lambda}
There exist $\underline a>0$ and $\kappa>0$ 
such that, for any $x,\tilde x\in\dR^d$ and $i,j\in E$, 
\[
a(x,i,j)\geq \underline a
\quad\text{and}\quad
\sum_{j\in E}\ABS{a(x,i,j)-a(\tilde x,i,j)}\leq \kappa\ABS{x-\tilde x}.
\]
\end{assumption}
The lower bound condition insures that the second --- discrete ---  coordinate of $Z$ changes 
often enough (so that the second coordinates of two independent copies of $Z$ coincide 
sufficiently often). This is a rather strong condition that can be relaxed in certain situations.
If for example  $a(x,i,j)=\lambda(x,i)P_{ij}$ as in \eqref{eq:gen-inf},  Lipschitz continuous 
$\lambda$ bounded  from below and  an aperiodic $P$ are sufficient to get a similar rate 
of convergence.  
\begin{assumption}\label{as:dissip}
Assume that there exists  $\alpha>0$ such that, 
\begin{equation}\label{eq:dissip}
\DP{x-\tilde x,F^i(x)-F^i(\tilde x)}\leq -\alpha\ABS{x-\tilde x}^2,
\quad
x,\tilde x\in\dR^d,\ i\in E.
\end{equation}
\end{assumption}
Assumption \ref{as:dissip} ensures that, for any $i\in E$, 
\[
\ABS{\varphi^i_t(x)-\varphi^i_t(\tilde x)}\leq  e^{-\alpha t}\ABS{x-\tilde x},
\quad
x,\tilde x\in \dR^d.
\]  
As a consequence, the vector fields $F^i$ has exactly one critical point 
 $\sigma(i)\in\dR^d$. Moreover it is exponentially stable since, for any 
 $x\in\dR^d$, 
\[
  \ABS{\varphi^i_t(x)-\sigma(i)}\leq e^{-\alpha t}\ABS{x-\sigma(i)}.
\]
In particular, $X$ cannot escape from a sufficiently large ball (this implies that the 
(continuous) functions $a(\cdot,i,j)$ are also bounded from above along the trajectories of $X$). 
More precisely, the following estimate holds. 
\begin{lemma}\label{lem:compact}
 Under Assumptions \ref{ass:lambda} and \ref{as:dissip}, the process $Z$ 
 cannot escape from the compact set $\bar B(0,r)\times E$ 
 where $\bar B(0,r)$ is the (closed) ball centered in $0\in\dR^d$ with radius 
 $r$ given by 
 \begin{equation}\label{eq:defr}
 r=\frac{\max_{i\in E}\ABS{F^i(0)}}{\alpha}.
\end{equation}
 Moreover, if $\ABS{X_0}>r$ then 
 \[
\PAR{\ABS{X_t}^2-r^2}^+\leq e^{-\alpha t}\PAR{\ABS{X_0}^2-r^2}^+.
\]
In particular the support of any invariant measure is included in 
$\bar B(0,r)$. 
 \end{lemma}
Let us now state our main result which establishes the quantitative 
exponential ergodicity of the process $Z$ under Assumptions 
\ref{ass:lambda} and \ref{as:dissip}.
\begin{theorem}\label{th:nonconstant}
Assume that Assumptions \ref{ass:lambda} and \ref{as:dissip} hold
and that the supports of $\mu_0$ and $\tilde \mu_0$ are included 
in the ball $\bar B(0,r)$ where $r$ is given by \eqref{eq:defr}. Then there exist 
positive constants $c$ and $\gamma$ such that
\[
\cW_1\PAR{\eta_t,\tilde \eta_t}\leq %
(1+2r)  \PAR{1+c t}\exp\PAR{-\frac{\alpha }{1+\alpha/\gamma} t}
\]
where $\alpha$ is derived from \eqref{eq:dissip}.
\end{theorem}
The constants obtained in the proof are the following
\[
\gamma=
\frac{(\alpha+b)-
\sqrt{(\alpha+b)^2-4bp\alpha }}{2}
\quad \text{and}\quad%
c=\frac{\alpha}{\alpha+\gamma} 
\frac{ep\alpha b }
{\sqrt{(\alpha+b)^2-4bp\alpha }},
\]
with $p=e^{-2r\kappa/\alpha}$ and $e=\exp(1)$ and where $b$ depends on the 
coalescence time of two independent processes defined on $E$ as the second 
coordinates of independent copies of $Z$. To obtain this result we couple two 
paths of our process and compare their distance to a real-valued process that 
can pass instantly (less and less often) from small to large values (2$r$+1). This 
may seem rough, but in the general case, nothing much better can be done : if 
one of the flows is very strongly attractive, two paths may indeed very rapidly 
diverge. For particular examples, or under additional assumptions on the flows, 
it must be possible to get better rates.

\begin{corollary}
Under the hypotheses of Theorem~\eqref{th:nonconstant}, the process $Z$
admits a unique invariant measure $\eta_\infty$ and 
\[
 \cW_1\PAR{\eta_t,\eta_\infty}\leq %
(1+2r)  \PAR{1+c t}\exp\PAR{-\frac{\alpha }{1+\alpha/\gamma} t}.
\]
\end{corollary}

Section~\ref{sec:constant}
is dedicated to the proof of Theorem \ref{th:constant}. 
Theorem~\ref{th:nonconstant} is established in Section~\ref{sec:nonconstant}. 
As an example of the applicability of this result, we show
in Section~\ref{sec:Example} that the stochastic version of the Morris--Lecar model
introduced in~\cite{WTP10} satisfies our assumptions, so that explicit bounds 
may be obtained for its convergence to equilibrium. 

\section{Constant jump rates}
\label{sec:constant}

The aim of this section is to prove Theorem \ref{th:constant}. 
Assumption \ref{as:ymarkov} ensures that ${(I_t)}_{t\geq 0}$ is an irreducible 
Markov process on the finite space $E$. Its generator $A$ is the matrix defined 
by $A(i,i)=-\lambda(i)$ and $A(i,j)=\lambda(i)P(i,j)$ for 
$i\neq j$. Let us denote by $\nu_\infty$ its unique invariant probability 
measure. The study of the long time behavior of $I$ is classical: since $I$ 
takes its values in a finite set, it is quite simple to construct a coalescent 
coupling of two processes starting from different points. 

\begin{lemma}[\cite{MR1490046}]\label{lem:rho}
 If Assumption \ref{as:ymarkov} holds then there exists $\rho>0$ such that 
 for any $i,j\in E$, 
\begin{equation}\label{eq:saloff}
 \dP(T>t\vert I_0=i,\ \tilde I_0=j)\leq e^{-\rho t} 
\end{equation}
 where ${(I_t)}_{t\geq 0}$ and ${(\tilde I_t)}_{t\geq 0}$ are two independent 
 Markov processes with infinitesimal generator $A$ starting respectively 
 at $i$ and $j$ and $T=\inf\BRA{t\geq 0\, :\, I_t=\tilde I_t}$ is the first 
 intersection time. 
 \end{lemma}

\begin{remark}
If $E=\BRA{1,2}$, then the first intersection time is distributed as an 
exponential random variable with parameter $\lambda(1)+\lambda(2)$ and 
Equation~\eqref{eq:saloff} holds with $\rho=\lambda(1)+\lambda(2)$. 
\end{remark}

The proof of Theorem~\ref{th:constant} is made of two steps. Firstly we 
couple two processes starting respectively from $(x,i)$ and $(\tilde x,i)$ 
to get a simple estimate as time goes to infinity. Then we use this estimate 
and Lemma~\ref{lem:rho} to manage the general case.

\subsection{Moments estimates}

In this section we prove Lemma \ref{lem:moment2} and get an $L^p$ 
estimate for $\ABS{X_t}$. For any $p\geq 2$ and $\varepsilon>0$,
\begin{align*}
 \frac{d}{dt}\ABS{X_t}^p&=p\ABS{X_t}^{p-2}\DP{X_t,F^{I_t}(X_t)}\\
 &=p\ABS{X_t}^{p-2}\DP{X_t,F^{I_t}(X_t)-F^{I_t}(0)}
 +p\ABS{X_t}^{p-2}\DP{X_t,F^{I_t}(0)}\\
&\leq -(p\alpha(I_t)-\varepsilon)\ABS{X_t}^p+C(p,\varepsilon).
\end{align*}
Thanks to Gronwall's Lemma, we get that 
\[
\dE\PAR{\ABS{X_t}^p}\leq C(p,\varepsilon)\int_0^t\!
\dE\PAR{e^{-\int_s^t\!(p\alpha(I_u)-\varepsilon)\,du}}\,ds
+\dE\PAR{\ABS{X_0}^p}\dE\PAR{e^{-\int_0^t\!(p\alpha(I_u)-\varepsilon)\,du}}.
\]
\begin{remark}
 A similar estimate can be obtained for $p\in[1,2)$ using a regularization of 
 the application $x\mapsto \ABS{x}^p$. 
\end{remark}
The right hand side depends only on $\dE(\ABS{X_0}^p)$ and 
${(I_s)}_{0\leq s\leq t}$. 
Thus, it is sufficient to investigate the behavior of $e{(p,t)}$ 
defined for any $t\geq 0$ by 
\[
e{(p,t)}=%
\max_{i\in E}\dE_i\PAR{\exp\PAR{-\int_0^t\!p\alpha(I_u)\,du}}.
\]
This study has been already performed in~\cite{bgm}. 
Let us state the precise result. We denote by $A_p$ the matrix $A-pB$ 
where $A$ is the infinitesimal generator of $I$ and $B$ is the 
diagonal matrix with diagonal $(\alpha(1),\ldots,\alpha(n))$ 
and associate to $A_p$ the quantity
\begin{equation}\label{eq:etap}
\theta_p:=-\max_{\gamma\in\mathrm{Spec}(A_p)} \mathrm{Re\ }\gamma. 
\end{equation}
The long time behavior of $e{(p,t)}$ is characterised by $\theta_p$ as 
follows. For any $p>0$, there exist 
$0<C_{1}(p) <1< C_{2}(p) <+\infty$ such that, for any  
any $t>0$,
\begin{equation}
\label{eq:conv-expo}
C_{1}(p)e^{-\theta_{p}t} \leq 
e{(p,t)}
\leq C_{2}(p)e^{-\theta_{p}t}.
\end{equation}
Moreover the following dichotomy holds:
 \begin{itemize}
 \item if $\underline \alpha\geq 0$, then $\theta_p>0$ for all $p>0$,
 \item if $\underline\alpha< 0$, there is 
$\kappa\in(0,\min\{A_{ii}/\alpha(i)\,:\, \alpha(i)<0\})$ such that 
$\theta_p>0$ for $p<\kappa$ and $\theta_p<0$ for $p>\kappa$.
\end{itemize}
See \cite{bgm} for further details.

If $p<\kappa$, then $t\mapsto \dE(\ABS{X_t}^p)$ is bounded as soon as 
 $\dE(\ABS{X_0}^p)$ is finite. This concludes the proof of Lemma~\ref{lem:moment2}.

\subsection{Convergence rate}

Let us now get the upper bound for the Wasserstein distance $W_p$ for 
some $p<\kappa$. Assume firstly that the initial law are two Dirac 
masses at $(x,i)$ and $(\tilde x,i)$. It is easy to construct a good coupling 
of the two processes $(X,I)$ and $(\tilde X,\tilde I)$: since the jump rates 
of $I$ do not depend on $X$, one can choose $I$ and $\tilde I$ equal! 
As a consequence, for any $p\geq 2$, 
\begin{align*}
 \frac{d}{dt}\vert X_t-\tilde X_t\vert ^p&=
 p\vert X_t-\tilde X_t\vert ^{p-2}
 \DP{X_t-\tilde X_t,F^{I_t}(X_t)-F^{I_t}(\tilde X_t)}\\
 &\leq -p\alpha (I_t) \vert X_t-\tilde X_t\vert ^p.
\end{align*}
As a consequence, 
\begin{align*}
\dE\PAR{\vert X_t-\tilde X_t\vert ^p}&\leq  
\dE_i\PAR{\exp\PAR{-p\int_0^t\! \alpha (I_s)\,ds}}\ABS{x-\tilde x}^p\\
&\leq e^{-\theta_p t}\ABS{x-\tilde x}^p.
\end{align*}

Let us now turn to a general initial condition. Choose $(x,i)$ 
and $(\tilde x,j)$ in $\dR^d\times E$ and consider 
the following coupling: the two processes evolve 
independently until the intersection time $T$ of the second coordinates. 
Then, $I$ and $\tilde I$ are chosen to be equal for ever: 
\[
\dP(I_t\neq \tilde I_t)=\dP(T>t)\leq e^{-\rho t}
\]
by~\eqref{eq:saloff}. This term will turn to be negligible in $\cW_p(\eta_t,\tilde\eta_t)$. 
Now fix $t>0$ and $\beta\in(0,1)$ and decompose:
\[
\dE\PAR{\vert X_t-\tilde X_t\vert ^p}=
\dE\PAR{\vert X_t-\tilde X_t\vert ^p\ind_\BRA{T>\beta t}}+
\dE\PAR{\vert X_t-\tilde X_t\vert ^p\ind_\BRA{T\leq \beta t}}.
\]
Choose $q\in(p,\kappa)$ and define $r=q/p$ and $s$ as the conjugate of $r$. 
By H\"older's inequality and Lemma~\ref{lem:moment2}, we get
\begin{align*}
\dE\PAR{\vert X_t-\tilde X_t\vert ^p\ind_\BRA{T>\beta t}}&\leq 
\dE\PAR{\vert X_t-\tilde X_t\vert ^q}^{p/q} \dP\PAR{T\geq \beta t}^{1/s}\\
&\leq 2^p M(q,m)^{p/q} e^{-(\beta \rho/s)t }. 
\end{align*}
Moreover,
\begin{align*}
 \dE\PAR{\vert X_t-\tilde X_t\vert ^p\ind_\BRA{T\leq \beta t}}&=
 \dE\PAR{\vert X_T-\tilde X_T\vert ^p \dE_{I_T}\PAR{\exp\PAR{-p\int_T^t\!\alpha(I_s)\,ds}}\ind_\BRA{T\leq \beta t}}\\
 &\leq 2^p M(p,m) C_2(p)e^{-\theta_p(1-\beta)t}.
\end{align*}
At last, one has to optimize over $\beta\in(0,1)$. With 
\[
\beta=\frac{\theta_p}{\theta_p+\rho/s},
\] 
one has 
\[
\dE\PAR{\vert X_t-\tilde X_t\vert ^p}\leq 
2^{p+1}M(q,m)^{p/q}C_2(p) \exp\PAR{-\frac{\rho/s}{\theta_p+\rho/s}\theta_pt}.
\]
This concludes the proof of Theorem~\ref{th:constant}.

\section{Non constant jump rates}
\label{sec:nonconstant}

Let us now turn to the proof of Theorem \ref{th:nonconstant}. 
In this section we do not assume that the jump rates depend only 
on the discrete component. Thus, the coupling is more subtle 
since once $I$ and $\tilde I$ are equal, they can go apart with 
a positive rate. The main idea is the following.
If $I$ and $\tilde I$ are equal, the distance between $X$ and $\tilde X$ 
decreases exponentially fast and then it should be more and more 
easier to make the processes $I$ and $\tilde I$ jump simultaneously 
(since the jump rates are Lipschitz functions of $X$). This idea has been 
used in a different framework in \cite{CMP,BCGMZ}.

This section is organized as follows. Firstly we prove 
Lemma~\ref{lem:compact} that ensures that the process 
$X$ cannot escape from a sufficiently large ball. In particular, the support 
of the invariant law of $X$ is included in this ball. Then we construct the 
coupling of two processes $Z=(X,I)$ and $Z=(\tilde X,\tilde I)$ driven by the 
same infinitesimal generator \eqref{eq:gen-inf} with different initial conditions. 
At last we compare the distance between $Z$ and $\tilde Z$ to an companion
process that goes to 0 exponentially fast.  

\begin{proof}[Proof of Lemma \ref{lem:compact}]
  Setting $\tilde x=0$ in \eqref{eq:dissip} ensures that, for 
 $\varepsilon\in(0,\alpha)$,
\[
\DP{F^i(x),x}\leq -\alpha\ABS{x}^2+\DP{F^i(0),x}\leq %
-(\alpha-\varepsilon)\ABS{x}^2+M/(4\varepsilon),
\]
if $M=\max_{i\in E}\ABS{F^i(0)}^2$. In other 
words, 
\[
\ABS{X_t}^2-\ABS{X_s}^2=
\int_s^t\! 2 \DP{F^{I_u}(X_u),X_u}\,du\leq 
-2(\alpha-\varepsilon) \int_s^t\ABS{X_u}^2\,du+\frac{M}{2\varepsilon}(t-s). 
\]
As a consequence, 
\[
\ABS{X_t}^2\leq 
\frac{M}{4\varepsilon(\alpha-\varepsilon)}(1-e^{-2(\alpha-\varepsilon) t})
+\ABS{X_0}^2e^{-2(\alpha-\varepsilon) t}.
\]
Choosing $\varepsilon=\alpha/2$ ensures that 
\[
\PAR{\ABS{X_t}^2-\frac{M}{\alpha^2}}^+
\leq e^{-\alpha t}\PAR{\ABS{X_0}^2-\frac{M}{\alpha^2}}^+.
\]
In particular, $X$ cannot escape from the centered closed ball with 
radius $r=\sqrt{M}/\alpha$.
\end{proof}

\subsection{The coupling}

Let us construct a Markov process on $(\dR^d\times E)^2$ with 
marginals driven by \eqref{eq:gen-inf} starting respectively from 
$(x,i)$ and $(\tilde x,j)$. This is done \emph{via} its infinitesimal 
generator $\gen$ which is defined as follows: 
\begin{itemize}

\item if $i\neq j$ 
\begin{align*}
\gen f(x,i,\tilde x,j)=&%
\DP{F^i(x),\nabla_xf(x,i,\tilde x,j)}+\DP{F^j(\tilde x),\nabla_{\tilde x}f(x,i,\tilde x,j)}\\
&+\sum_{i'\in E}a(x,i,i') (f(x,i',\tilde x,j)-f(x,i,\tilde x,j))\\
&+\sum_{j'\in E}a(\tilde x,j,j') (f(x,y,\tilde x,j')-f(x,y,\tilde x,j)). 
\end{align*}

\item if $i=j$:
\begin{align*}
\gen f(x,i,\tilde x, j)=&%
\DP{F^i(x),\nabla_xf(x,i,\tilde x,i)}%
+\DP{F^i(\tilde x),\nabla_{\tilde x}f(x,i,\tilde x,i)}\\%
&+ \sum_{i'\in E}(a(x,i,i')\wedge a(\tilde x,i,i')) (f(x,i',\tilde x,i')-f(x,i,\tilde x,i))\\
&+ \sum_{i'\in E}(a(x,i,i')-a(\tilde x,i,i'))^+\ (f(x,i',\tilde x,i)-f(x,i,\tilde x,i))\\
&+ \sum_{i'\in E}(a(x,i,i')-a(\tilde x,i,i'))^-\ (f(x,i,\tilde x,i')-f(x,i,\tilde x,i)),
\end{align*}
\end{itemize}
where $(\cdot)^+$ and $(\cdot)^-$ stand respectively for the positive and negative 
parts. Notice that if $f$ depends only on $(x,i)$ or on $(\tilde x,j)$, then 
$\gen f=Af$. Let us explain how this coupling works. When $I$ and $\tilde I$ 
are different, the two processes $(X,I)$ and $(\tilde X,\tilde I)$ 
evolve independently. If $I=\tilde I$ then two jump processes are in 
competition: a single jump \emph{vs} two simultaneous 
jumps. The rate of arrival of a single jump equals to 
$\sum_{i'\in E}\ABS{a(x,i,i')-a(\tilde x,i,i')}$. It is bounded above by 
$\kappa \ABS{x-\tilde x}$. The rate of arrival of a simultaneous jump is 
given by $\sum_{i'\in E}(a(x,i,i')\wedge a(\tilde x,i,i'))$. 


Assume firstly that $X_0$ and $\tilde X_0$ belong to the ball 
$\bar B(0,r)$ where $r$ is given by \eqref{eq:defr}. Let us define, 
for any $t\geq 0$,
\[
\Delta_t=\vert X_t-\tilde X_t\vert +\ind_\BRA{I_t\neq \tilde I_t}.
\]
The process ${(\Delta_t)}_{t\geq 0}$ is not 
Markovian. Nevertheless, as long as $I=\tilde I$, $\Delta$ decreases 
with an exponential rate which is greater than $\alpha$. If a single jump occurs, then 
$\Delta$ is increased by 1 and it can continuously increase (since the continuous parts are driven 
by two different vector fields). Nevertheless $\Delta$ is bounded above by $D+1$ with $D=2 r$. 
At the next coalescent time $T_c$ of two independent copies of $I$, $\Delta$ jumps to $\Delta-1$ and 
then decreases exponential fast once again (as long as the discrete components coincide). 
There exists $b>0$ such that $T_c$ is 
(stochastically) smaller than $\cE(b)$ (for example, if $E=\BRA{0,1}$, then $T_c$ is 
equal to the minimum of the jump times of the two independent processes which 
are both stochastically smaller than a random variable of law $\cE(\underline a)$ 
and $T_c$ is stochastically smaller than $\cE(2\underline a)$). Then 
$\dE(\Delta_t)\leq \dE(U_t)$ where the Markov process ${(U_t)}_{t\geq 0}$ 
on $[0,D]\cup\BRA{D+1}$ is driven by the infinitesimal generator  
\[
Gf(x)=
\begin{cases}
\displaystyle{ -\alpha xf'(x)+\kappa x(f(D+1)-f(x))}
&\text{if }x\in [0,D],\\ 
b (f(D)-f(D+1)) 
 &\text{if }x=D+1. 
\end{cases}
\]

\subsection{The companion process}

\begin{theorem}
For any $t\geq 0$,
\begin{equation}\label{eq:estimate}
\dE\PAR{U_t\vert U_0=D}
\leq  
\PAR{D+(D+1)
\PAR{\frac{p\alpha be}{\sqrt{(\alpha+b)^2-4p\alpha b}}}
\frac{\alpha t}{\alpha +\gamma}}
\exp\PAR{-\frac{1}{1+\alpha/\gamma}\alpha t}
 \end{equation}
where 
\[
p=e^{-D \kappa/\alpha }%
\quad \text{and}\quad%
\gamma=\frac{(\alpha +b)-\sqrt{(\alpha +b)^2-4p\alpha b}}{2}
=\frac{(\alpha +b)-\sqrt{(\alpha -b)^2+4(1-p)\alpha b}}{2}.
\]
\end{theorem}

\begin{remark}
 If $\alpha$ goes to $\infty$, then $\gamma$ goes to d whereas 
 $\gamma\sim p\alpha /b$ if $b$ goes to $\infty$. 
\end{remark}

\begin{proof}
Starting from $D+1$, the process $U$ jumps after 
a random time with law $\cE(b)$ to $d$ and then goes to zero 
exponentially fast until it (possibly) goes back to $D+1$. 
The first jump time $T$ starting from $D$ can be constructed 
as follows: let $E$ be a random variable with exponential law $\cE(1)$. Then 
\[
T\overset{\cL}{=}
\begin{cases}
\displaystyle{-\frac{1}{\alpha}\log\PAR{1-\frac{\alpha E}{D \kappa}}}& \text{if }
\displaystyle{E<\frac{D \kappa}{\alpha}}, \\
 +\infty &\text{otherwise.} 
\end{cases}
\]
Indeed, conditionally on $\BRA{U_0=D}$, 
\[
\int_0^t\!\lambda (V_s)ds=\int_0^t\!d \kappa e^{-\alpha s}\,ds%
=\frac{D \kappa}{\alpha}(1-e^{-\alpha t}). 
\]
In other words, the cumulative distribution function $F_T$ of $T$ 
is such that, for any $t\geq 0$, 
\[
1-F_T(t)=\dP(T>t)=\exp\PAR{-\frac{D \kappa}{\alpha}(1-e^{-\alpha t})}.
\]
Let us define $p=e^{-D \kappa/\alpha}$. The law of $T$ is 
the mixture with respective weights $p$ and $1-p$ of a Dirac mass 
at $+\infty$ and a probability measure on $\dR$ with density 
\begin{equation}\label{eq:densite}
f\,:\,t\mapsto f(t)=\frac{D \kappa}{1-p}%
 e^{-\alpha t}e^{-\frac{D \kappa}{\alpha }(1-e^{-\alpha t})}\ind_{(0,+\infty)}(t) 
\end{equation}
and cumulative distribution function 
\[
F\,:\,t\mapsto F(t)=%
 \PAR{\frac{1-\exp\PAR{-\frac{D \kappa}{\alpha}(1-e^{-\alpha t})}}
 {1-\exp\PAR{-\frac{D \kappa}{\alpha}}}}\ind_{(0,+\infty)}(t).
\]

Starting at $D$, $U$ will return to $D$ with probability $1-p$. The Markov 
property ensures that the number $N$ of returns of $U$ to $D$ is a 
random variable with geometric law with parameter $p$. The length 
of a finite loop from $D$ to $D$ can be written as the sum $S+E$ 
where the law of $S$ has the density function $f$ given in \eqref{eq:densite}, 
the law of $E$ is the exponential measure with parameter $b$ 
and $S$ and $E$ are independent. 

\begin{lemma}\label{lem:comp-expo}
The variable $S$ is stochastically smaller than an exponential random 
variable with parameter $\alpha$ \emph{i.e.} for any $t\geq 0$, 
$F(t)\geq F_\alpha(t)$ where $F_\alpha(t)=(1-e^{-\alpha t})\ind_\BRA{t>0}$.
\end{lemma}

\begin{proof}
For any $t\geq 0$, 
\[
1-F(t)=%
\frac{\exp\PAR{\frac{D \kappa}{\alpha}e^{-\alpha t}}-1}{\exp\PAR{\frac{D \kappa}{\alpha}}-1} %
\leq e^{-\alpha t}=1-F_\alpha(t).
\]
This ensures the stochastic bound. 
\end{proof}
As a consequence, the Laplace transform $L_S$ of $S$ with density $f$ 
is smaller than the one of an exponential variable with parameter $\alpha$: 
for any $s<\alpha$, 
\[
L_S(s) \leq \frac{\alpha}{\alpha-s}.
\]
If $L_e$ is the Laplace transform of $S+E$, then, for any 
$s<\alpha \wedge b$, we have
\[
L_e(s)\leq \frac{\alpha}{\alpha-s}\frac{b}{b-s}.
\]
Let us denote by $H$ the last hitting time of $d$ \emph{i.e.} the last jump 
time of $U$ and by $L$ its Laplace transform. Let us introduce 
$N\sim\cG(p)$, ${(S_i)}_{i\geq 1}$ with density $f$ and 
${(E_i)}_{i\geq 1}$ with law $\cE(b)$. All the random variables are 
assumed to be independent. Then  
\[
H\overset{\cL}{=}\sum_{i=1}^N (S_i+E_i). 
\]
Classically, for any $s\in\dR$ such that 
$(1-p)L_e(s)<1$, one has 
\[
 L(s)=%
\dE\PAR{e^{s H}}=\frac{pL_e(s)}{1-(1-p)L_e(s)}%
=\frac{p}{1-p}\PAR{\frac{1}{1-(1-p)L_e(s)}-1}.
\]
Let us denote by 
\[
\gamma=\frac{(\alpha+b)-\sqrt{(\alpha+b)^2-4p\alpha b}}{2}
\quad\text{and}\quad
\tilde \gamma=\frac{(\alpha+b)+\sqrt{(\alpha+b)^2-4p\alpha b}}{2}
\]
the two roots of $\xi^2-(\alpha+b)\xi+p\alpha b=0$. Notice that 
$\gamma<\alpha\wedge b<\tilde \gamma$. For any 
$s<\gamma$, one has $(1-p)L_e(s)<1$ and 
\begin{equation}\label{eq:maj-laplace}
L(s)\leq \frac{p \alpha b}{(\gamma-s)(\tilde\gamma-s)}
\leq \frac{p \alpha b}{\tilde\gamma-s}\frac{1}{\gamma-s}.
\end{equation}

Let us now turn to the control of $\dE\PAR{U_t\vert U_0=D}$. The idea is 
to discuss whether $H>\beta t$ or not for some $\beta\in(0,1)$ (and then to 
choose $\beta$ as good as possible):
\begin{itemize}
\item if $H<\beta t$, then $U_t\leq e^{-(1-\beta)\alpha t}$,
\item the event $\BRA{H\geq \beta t}$ has a small probability 
for large $t$ since $H$ has a finite Laplace transform on a 
neighbourhood of the origin. 
\end{itemize}
For any $\beta\in(0,1)$ and 
$s>0$, 
\begin{align}
\dE\PAR{U_t\vert U_0=D}&=\dE\PAR{U_t\ind_\BRA{T\leq \beta t}}+%
\dE\PAR{U_t\ind_\BRA{T> \beta t}}\nonumber \\
&\leq De^{-(1-\beta)\alpha t}+(D+1)L(s) e^{-s \beta t}. 
 \label{eq:majo}
 \end{align}
 From Equation \eqref{eq:maj-laplace}, we get that, for any 
 $s<\gamma$, $\log L(s)-\beta t s \leq h(s)$
where
\[
h(s)=\log \PAR{\frac{p\alpha b}{\tilde \gamma-\gamma}}
-\log(\gamma-s)-\beta t s. 
\]
The function $h$ reaches its minimum at $s(t)=\gamma-(\beta t)^{-1}$
and 
\[
h(s(t))=%
\log \PAR{\frac{p\alpha b}{\tilde \gamma-\gamma}}+%
\log(\beta t)+1-\gamma\beta t.
\]
For $t>0$ and $\beta\in (0,1)$, choose $s(t)=\gamma-(\beta t)^{-1}$ 
in \eqref{eq:majo} to get 
\begin{align*}
\dE\PAR{U_t\vert U_0=D}&\leq 
De^{-(1-\beta)\alpha t}+(D+1) e^{h(\gamma(t))}\\
&\leq De^{-(1-\beta)\alpha t}+(D+1) 
\PAR{\frac{p\alpha be}{\tilde \gamma-\gamma}}%
\beta t e^{-\gamma \beta t}.
\end{align*}
At last, one can choose $\beta=\alpha (\alpha+\gamma)^{-1}$ in order to 
have $(1-\beta)\alpha=\gamma \beta$. This ensures that 
\[
\dE\PAR{U_t\vert U_0=D}\leq  
\PAR{D+(D+1) \PAR{\frac{p\alpha be}{\tilde \gamma-\gamma}}
\frac{\alpha t}{\alpha+\gamma}}
\exp\PAR{-\frac{\alpha\gamma}{\alpha+\gamma}t}. 
\]
Replacing $\tilde\gamma-\gamma$ by its expression as a 
function of $\alpha$, $b$ and $p$ provides \eqref{eq:estimate}.
\end{proof}

\section{Example} 
\label{sec:Example}
The Morris--Lecar model introduced in \cite{ML81} describes the evolution in time of the electric 
potential $V(t)$ in a neuron. The neuron exchanges different ions with its environment via ion 
channels which may be open or closed. In the original ---~deterministic~--- model, the proportion 
of open channels of different types are coded by two functions $m(t)$ and $n(t)$, and the 
three quantities $m$, $n$ and $V$ evolve through the flow of an ordinary differential equation. 

Various stochastic versions of this model exist. Here we focus on a model described in~\cite{WTP10}, 
to which we refer for additional information. This model is motivated by the fact that
 $m$ and $n$, being proportions of open channels, are better 
coded as discrete variables. More precisely, we fix a large integer $K$ (the total 
number of channels) and define a PDMP $(V,u_1,u_2)$ with values in 
$\dR \times \{0,1/K,2/K \dots, 1\}^2$ as follows. 

Firstly, the potential $V$ evolves according to
\begin{equation}
  \frac{dV(t)}{dt} = \frac{1}{C} \left(I - \sum_{i=1}^3 g_iu_i(t) (V - V_i) \right)
  \label{eq=evolV}
\end{equation}
where $C$ and $I$ are positive constants (the capacitance and input current), the $g_i$ and $V_i$ 
are positive constants (representing conductances and equilibrium potentials for different types 
of ions), $u_3(t)$ is equal to $1$ and $u_1(t)$, $u_2(t)$ are the (discrete) proportions of open 
channels for two types of ions. 

These two discrete variables follow birth-death processes on $\{0, 1/K, \ldots,1\}$ with 
birth rates $\alpha_1$, $\alpha_2$ and death rates $\beta_1$, $\beta_2$  that depend on the 
potential $V$: 
\begin{equation}
  \begin{aligned}
  \alpha_i(V) &= c_i \cosh\left( \frac{V - V'_i}{2V''_i}\right) \left(1 + \tanh\left(\frac{V - V'_i}{V''_i}\right)
  \right) \\
  \beta_i(V) &= c_i \cosh\left( \frac{V - V'_i}{2V''_i}\right) \left(1 - \tanh\left(\frac{V - V'_i}{V''_i}\right)
  \right)
  \end{aligned}
  \label{eq=rates}
\end{equation}
where the $c_i$ and $V'_i$, $V''_i$ are constants. 

Let us check  that our main result can be applied in this example. 
Formally the process is a PDMP with $d=1$ and the finite set $E = \{0, 1/K,\ldots,1\}^2$. 
The discrete process $(u_1,u_2)$ plays the role of the index $i\in E$, and the fields
$F^{(u_1,u_2)}$ are defined (on $\dR$) by \eqref{eq=evolV} by setting $u_1(t) = u_1$, 
$u_2(t) = u_2$. 

The constant term $u_3g_3$ in \eqref{eq=evolV} ensures that the uniform dissipation property
\eqref{eq:dissip} is satisfied: for all $(u_1,u_2)$, 
\begin{align*}
  \DP{ V - \tilde{V}, F^{(u_1,u_2)}(V) - F^{(u_1,u_2)}(\tilde{V})}
  &= - \frac{1}{C}\sum_{i=1}^3 u_ig_i(V-\tilde{V})^2 \\
  &\leq - \frac{1}{C}u_3g_3 (V-\tilde{V})^2.
\end{align*}

The Lipschitz character and the bound from below on the rates are not immediate. 
Indeed the jump rates \eqref{eq=rates} are not bounded from below if $V$ is allowed to 
take values in $\dR$. 

However, a direct analysis of \eqref{eq=evolV} shows that $V$ is essentially bounded~: 
all the fields $F^{(u_1,u_2)}$ point inward at the boundary of the (fixed) line segment
$\mathcal{S} = [0,\max(V_1,V_2, V_3 + (I+1)/g_3u_3)]$, so if $V(t)$ starts
in this region it cannot get out. The necessary bounds all follow by compactness, since
$\alpha_i(V)$ and $\beta_i(V)$ are $\mathcal{C}^1$ in $\mathcal{S}$ and  strictly positive. 


\bibliographystyle{amsplain}
\bibliography{ecp-quantitatif}

\ACKNO{We deeply thank the referee for his/her quick and constructive report. 
FM and PAZ thank MB for his kind hospitality 
and his coffee breaks. We acknowledge financial support from the Swiss National Foundation 
Grant  FN 200021-138242/1 and the French ANR projects EVOL and ProbaGeo.}

\end{document}